\documentclass[12pt]{article}
\usepackage[margin=1in]{geometry} 
\usepackage{amsthm,amssymb,amsmath}
\usepackage{mathtools}
\usepackage{esvect}
\usepackage{yfonts}
\usepackage{tikz-cd}
\usepackage[bb=boondox]{mathalfa}
\usepackage{titling}
\usepackage{lipsum}
\usepackage{enumerate}
\usepackage{accents}

\let\seiresfb\bfseries\def\bfseries{\boldmath\seiresfb}
\let\seiresdm\mdseries\def\mdseries{\unboldmath\seiresdm}

\newcommand*{\Z}{\ensuremath{\mathbb{Z}}}  
\newcommand*{\R}{\ensuremath{\mathbb{R}}}  

\renewcommand*{\emptyset}{\varnothing}  

\renewcommand*{\iff}{\mathrel{\Leftrightarrow}}  

\renewcommand*{\implies}{\mathrel{\Rightarrow}}  

\newtheorem{theorem}{Theorem}
\newtheorem{definition}{Definition}
\newtheorem{remark}{Remark}

\newtheorem{lemma}{Lemma}
\newtheorem{claim}{Claim}
\newtheorem{corollary}{Corollary}
\newtheorem{proposition}{Proposition}
\newtheorem{conjecture}{Conjecture}

\usepackage{amsfonts}
\usepackage{stmaryrd}
\usepackage{yfonts}
\usepackage{cancel}

\begin{document}
\setlength{\parindent}{0cm}
\title{Minimal $2$-Spheres and Optimal Foliations in $3$-Spheres with Arbitrary Metric}
\author{Salim Deaibes}
\date{}
\maketitle
\noindent
\begin{abstract}
    In this paper, we prove that the $3$-sphere endowed with an arbitrary Riemannian metric either contains at least two embedded minimal $2$-spheres or admits an optimal foliation by $2$-spheres. This generalizes recent results by Haslhofer-Ketover (Duke Math. J. 2019), where the existence of optimal foliations and minimal $2$-spheres has been established under the additional assumption that the metric is generic. In light of recent examples by Wang-Zhou, where min-max for some non-bumpy metrics on the 3-sphere produces higher multiplicities, our results are in a certain sense sharp.
\end{abstract}
\section{Introduction}
A classical theorem of the geometry of surfaces is the Lusternik-Schnirelmann theorem pertaining to the existence of closed embedded geodesics in $2$-spheres:
\begin{theorem}[Lusternik-Schnirelmann \cite{lusternik}]\label{lusternik}
Let $(M^2,g)$ be a Riemannian manifold diffeomorphic to $\mathbb{S}^2$. Then, $M^2$ contains at least $3$ closed embedded geodesics.
\end{theorem}
The original proof had some gaps, which have since been corrected in several independent ways by \cite{ballmann,grayson,hass,jost,taimanov}. All proofs use a combination of variational methods (min-max) and a suitable curve-shortening procedure. The natural question is to what extent Theorem \ref{lusternik} generalizes to $3$-spheres. An outstanding conjecture in this direction is the following:
\begin{conjecture}\label{conjecture}
Let $(M^3,g)$ be a Riemannian manifold diffeomorphic to $\mathbb{S}^3$. Then, $M^3$ contains at least $4$ embedded minimal $2$-spheres.
\end{conjecture}
The motivation for this conjecture arises from Morse theory. If $\mathcal{S}$ denotes the space of embedded $2$-spheres in $\mathbb{S}^3$ (together with certain degenerations), then $\mathcal{S}/\partial \mathcal{S}$ is homotopy equivalent to $\R \mathbb{P}^4$ by Hatcher's theorem (Smale's conjecture, c.f. \cite{hatcher,bamler}). Hence, the corresponding relative cohomology ring is given by $H^*(\mathcal{S},\partial \mathcal{S},\Z_2) = \Z_2[\alpha]/(\alpha^5)$. We can then consider the associated area functional $\mathcal{A}:\mathcal{S} \to \R$. The non-trivial critical points of $\mathcal{A}$ are precisely the embedded minimal $2$-spheres in $(\mathbb{S}^3,g)$. Thus, formally applying Morse theory to the area functional, one expects to find at least $4$ embedded minimal $2$-spheres corresponding to the cohomology classes $\alpha$, $\alpha^2$, $\alpha^3$ and $\alpha^4$. 
We note that by a result of White \cite[Theorem 4.5]{white3}, the predicted number $4$ is sharp on certain perturbations of the round sphere.\\

In the 1980s, L. Simon and F. Smith developed a version of Almgren-Pitts min-max theory (now known as Simon-Smith min-max theory) for surfaces which allowed them to control the topology of the limit of minimizing sequences to prove a first result towards Conjecture \ref{conjecture}, namely:
\begin{theorem}[Simon-Smith \cite{smith}]\label{simon-smith}
Let $(M^3,g)$ be a Riemannian manifold diffeomorphic to $\mathbb{S}^3$. Then, $M$ contains at least one embedded minimal $2$-sphere.
\end{theorem}
The major difficulty in finding more than one solution is the phenomenon of multiplicity in min-max theory. Namely, the potential danger is that $k-$parameter min-max for the cohomology class $\alpha^k$ for $k=2,3,4$ may simply produce the same $2$-sphere, just with higher integer multiplicities.\\
We note that the multiplicity one conjecture for generic metrics has been proved in the Allen-Cahn and Almgren-Pitts setting in recent breakthroughs by Chodosh-Mantoulidis \cite{chodosh} and Zhou \cite{zhou}. However, establishing multiplicity one in the Simon-Smith setting, as well as in certain nongeneric situations, remains a major open problem.\\
Using degree theory, B. White improved the result of Simon-Smith under the additional assumption that the manifold has positive Ricci curvature (this curvature assumption is made to guarantee desirable compactness properties needed for degree theory), see \cite{white3}. It is clear, however, that $Ric_g>0$ is a restrictive assumption on metrics. Thus, the natural generalization of White's theorem is one which holds for 'almost all' Riemannian metrics (in some suitable sense). This is precisely what Haslhofer-Ketover proved:
\begin{theorem}[Haslhofer-Ketover \cite{haslhofer}]\label{haslhofer-ketover}
Let $(M^3,g)$ be a Riemannian manifold diffeomorphic to $\mathbb{S}^3$ endowed with a bumpy metric. Then, $M^3$ contains at least $2$ embedded minimal $2$-spheres. More precisely, exactly one of the following holds:
\begin{enumerate}
    \item $M$ contains at least one stable embedded minimal $2$-sphere, and at least two unstable embedded minimal $2$-spheres.
    \item $M$ contains no stable embedded minimal $2$-spheres and contains at least two unstable embedded minimal $2$-spheres.
\end{enumerate}
\end{theorem}

We recall that a metric $g$ is called $\emph{bumpy}$ if no immersed minimal hypersurfaces admit non-trivial Jacobi fields (i.e. functions which lie in the kernel of the stability operator).\\ A theorem of White (\cite[Theorem 2.2]{white4}) states that bumpy metrics are generic in the sense of Baire. The bumpiness assumption on a Riemannian metric is often needed for Morse-theoretic arguments. However, for many metrics that one encounters in practice, the bumpiness assumption either does not hold (e.g. for metrics with symmetries) or is unfeasible to check.\\

The proof by Haslhofer-Ketover combines techniques from min-max theory and mean curvature flow. More precisely (as reviewed in more detail below) in the difficult case when there is no stable minimal 2-sphere they used mean curvature flow with surgery to produce an optimal foliation by $2$-spheres:

\begin{theorem}[Haslhofer-Ketover \cite{haslhofer}]
Let $(M^3,g)$ be a Riemannian manifold diffeomorphic to $\mathbb{S}^3$ that does not contain any stable minimal 2-sphere. Then, $M^3$ contains admits an optimal foliation by $2$-spheres.
\end{theorem}

Here, an optimal foliation is, roughly speaking, a foliation by smooth embedded $2$-spheres that contains a central $2$-sphere realizing the one-parameter min-max width and such that all other leaves of the foliation have strictly smaller area, see \cite[Section 3]{haslhofer} for details. Haslhofer-Ketover then used their optimal foliation in combination with ideas from min-max theory, in particular the catenoid estimate \cite{ketover}, to produce a second embedded minimal $2$-sphere.\\
We remark that optimal foliations (of $3$-spheres or $3$-disks) are also of independent interest in other geometric problems, in particular the inverse problem for the area functional studied by Alexakis-Balehowsky-Nachman \cite{alexakis}, and the waist and Urysohn inequalities proved by Liokumovich-Maximo \cite{maximo} and Wang-Zhu \cite{wang}.\\

In the present paper, we investigate to what extent the approach by Haslhofer-Ketover generalizes to arbitrary Riemannian metrics without bumpyness assumption. Our main result shows that even without bumpyness assumption we always get \emph{either} the existence of at least two embedded minimal $2$-spheres \emph{or} the existence of an optimal foliation by $2$-spheres:

\begin{theorem}[Existence Theorem]\label{main1}
Let $(M^3,g)$ be a Riemannian manifold diffeomorphic to $\mathbb{S}^3$. Then, $M$ contains at least two embedded minimal $2$-spheres or admits an optimal foliation by $2$-spheres. More precisely, either $(M^3,g)$
\begin{enumerate}[(1)]
    \item contains a strictly stable embedded minimal $2$-sphere, in which case $M$ contains at least three embedded minimal $2$-spheres,
    \item or contains a degenerate stable embedded minimal $2$-sphere $\Sigma$. In this case, either
    \begin{enumerate}[(a)]
        \item $M$ contains infinitely-many embedded minimal $2$-spheres,
        \item or $\Sigma$ admits either
        \begin{enumerate}[(i)]
            \item a contracting neighbourhood, in which case $M$ contains at least three embedded minimal $2$-spheres,
            \item or a mixed neighbourhood, in which case $M$ contains at least two embedded minimal $2$-spheres,
            \item\label{bad} or an expanding neighbourhood, in which case $M$ admits an optimal foliation by $2$-spheres.
        \end{enumerate}
    \end{enumerate}
    \item or contains no stable embedded minimal $2$-spheres, in which case there exist at least two unstable embedded minimal $2$-spheres and an optimal foliation by $2$-spheres.
\end{enumerate}
\end{theorem}

\begin{remark}\label{intro}

We note that in Theorem \ref{main1}, we obtain the existence of at least two embedded minimal $2$-spheres in every case except for case \ref{bad}. In the nondegenerate case,  the optimal foliation $\{\Sigma_t\}$ satisfies:
\begin{align}\label{estimate}
    |\Sigma_t| < |\Sigma|-ct^2,
\end{align}
where $t \sim dist(\Sigma,\Sigma_t)$. The quadratic gain wins against the error term in the catenoid estimate (see \cite{ketover}), the latter being proportional to $t^2/\log t$. In the degenerate case, however, estimate (\ref{estimate}) need not hold (e.g. one could have that $|\Sigma_t| \sim |\Sigma|-ct^4$) - in particular, the error term in the catenoid estimate could win against the polynomial gain. In fact, Wang-Zhou (see \cite{zhou1}) recently showed that for certain degenerate metrics on $\mathbb{S}^3$, two-parameter min-max does indeed produce the first solution with multiplicity two.
\end{remark}

\subsection*{Review of the Haslhofer-Ketover Approach}
Before outlining our proof of Theorem \ref{main1}, we recall the main ideas of the proof of Theorem \ref{haslhofer-ketover} from \cite{haslhofer}. The authors distinguish two cases:
\begin{itemize}
    \item $M^3$ contains a stable embedded minimal $2$-sphere,
    \item $M^3$ does not contain a stable embedded minimal $2$-sphere.
\end{itemize}
In the first case, by performing a $1$-parameter min-max procedure in the two $3$-disks bounded by the stable (and thus strictly stable by bumpiness) minimal $2$-sphere, the authors obtain the existence of an embedded minimal $2$-sphere in the interior of each $3$-disk: in this case, the $3$-sphere contains at least $3$ embedded minimal $2$-spheres and the theorem is proven.\\
\newline
In the case that there does not exist a stable embedded minimal $2$-sphere, the proof is more delicate. The Simon-Smith existence theorem (Theorem \ref{simon-smith}) produces an embedded minimal $2$-sphere which is necessarily unstable: denote this sphere by $\Sigma$. By using the lowest eigenfunction of the stability operator $L_{\Sigma}$ of $\Sigma$, the manifold can be decomposed into $M=D^- \cup N(\Sigma) \cup D^+$, where $D^{\pm}$ are smooth $3$-disks with mean-convex boundary and $N(\Sigma)$ is a tubular neighbourhood of $\Sigma$ which is foliated by smooth $2$-spheres with mean-curvature vector pointing away from $\Sigma$.\\
Applying the theory of mean curvature flow with surgery, the authors show that there exist smooth foliations of $D^{\pm}$ by mean-convex embedded $2$-spheres. Using these foliations, one can then form an optimal foliation $\{\Sigma_t\}_{-1 \leq t \leq 1}$ of $M^3$ by $2$-spheres such that $\Sigma_0 = \Sigma$ and $|\Sigma_t| < |\Sigma_0|$ for $t \neq 0$.\\
The authors then construct a two-parameter family $\{\Sigma_{s,t}\}$, where, roughly speaking, $\Sigma_{s,t} = \Sigma_s \# \Sigma_t$ (the surface consisting of $\Sigma_s$ and $\Sigma_t$ connected by a thin neck). Using the catenoid estimate of Ketover-Marques-Neves \cite{ketover}, they then show that 
\begin{align}
    \sup_{s,t}|\Sigma_{s,t}| < 2|\Sigma|,
\end{align}
which guarantees that the minimal surface obtained by a $2$-parameter min-max procedure is not $\Sigma$ with multiplicity $2$. Finally, by a Lusternik-Schnirelmann argument, $\sup_{s,t}|\Sigma_{s,t}|$ is strictly bounded below by $|\Sigma|$, i.e. the min-max surface is also not $\Sigma$ - it must thus be a new embedded minimal $2$-sphere. \\ 
\subsection*{Outline of Our Proof}
We now outline our proof of Theorem \ref{main1}. Unlike Theorem \ref{haslhofer-ketover} which only had two cases to consider, we now have to consider the following three scenarios: 
\begin{itemize}
    \item $M^3$ contains a degenerate stable embedded minimal $2$-sphere but no strictly stable embedded minimal $2$-spheres,
    \item $M^3$ contains a strictly stable embedded minimal $2$-sphere,
    \item $M^3$ does not contain a stable embedded minimal $2$-sphere.
\end{itemize}
As we will seee, the main new difficulty is the scenario of $M$ containing a degenerate stable embedded minimal $2$-sphere but no strictly stable embedded minimal $2$-spheres.\\

Adapting a lemma from A. Song's proof of the Yau conjecture in \cite{song} to our setting, we show that if $\Sigma$ is a degenerate stable embedded minimal $2$-sphere, then either there exist infinitely-many embedded minimal $2$-spheres or $\Sigma$ admits a tubular neighbourhood which is either expanding, contracting or mixed (which will be made precise later on).\\ 

In the expanding case (i.e. $\Sigma$ has a tubular neighbourhood foliated by $2$-spheres with mean-curvature vector pointing away from $\Sigma$), we use the optimal foliation approach reviewed above. However, in order to emulate the optimal foliation argument of Haslhofer-Ketover, there are some subtleties that we have to address. Unlike in \cite{haslhofer}, if $\Sigma$ is not strictly stable, then it is not necessarily unstable (since the metric need not be bumpy). Thus, we prove a modified version of \cite[Lemma 3.2]{haslhofer}, which establishes the existence of an embedded minimal $2$-sphere with multiplicity $1$ realizing the $1$-width (i.e. the width associated to $1$-parameter sweepouts of $M^3$ by $2$-spheres) under the assumption that $M^3$ contains no stable embedded minimal $2$-spheres with contracting or mixed neighbourhoods. Using this, one can argue that the manifold admits an optimal foliation.\\ \newline
In the contracting case (i.e. $\Sigma$ has a tubular neighbourhood foliated by $2$-spheres whose mean-curvature vector points towards $\Sigma$), $M$ can be decomposed into $M=S^- \cup \Sigma \cup S^+$, where $S^{\pm}$ are smooth $3$-balls such that their closures have minimal boundary. In order to prove our theorem in this case, we seek to apply the Ketover-Liokumovich-Song min-max theorem \cite[Theorem 10]{lio} to each of the $3$-disks with boundary $\Sigma$, which is well-suited to min-max on compact manifolds with minimal boundary. To this end, we show that $\Sigma$ admits a Marques-Neves squeezing map: this ensures that some slice of any sweepout of $M$ by $2$-spheres starting at $\Sigma$ must have area greater than that of $\Sigma$. This, in turn, guarantees that the minimizing sequence produced converges to a minimal surface which lies in the interior of the $3$-disk and not just on the boundary (in particular, this surface cannot be $\Sigma$). In this case, we obtain at least $3$ embedded minimal $2$-spheres.\\ 

In the mixed case (i.e. $\Sigma$ has a tubular neighbourhood which has a contracting half and an expanding half), we apply the Ketover-Liokumovich-Song min-max theorem \cite[Theorem 10]{lio} to the contracting half to obtain a second embedded minimal $2$-sphere. \\ \newline

Finally, we observe (by means of a standard proposition) that the strictly stable case and unstable case are also covered by the argument outlined above. In the unstable case, however, one can use the optimal foliation argument from Haslhofer-Ketover to obtain a second embedded minimal $2$-sphere via the catenoid estimate.\\ \newline

This article is organized into two main sections: in Section 2, we recall the definitions pertaining to stability and discuss geometric neighbourhoods of degenerate stable minimal spheres. In Section 3, we give the proof of our main theorem (Theorem \ref{main1}).\\
\section*{Acknowledgments} I would like to deeply thank my supervisor Professor Robert Haslhofer
for suggesting this topic and for his support and patience throughout its completion. I would also like to thank Professor Yevgeny Liokumovich for several enlightening discussions.\\

\bigskip

\section{Stability and Neighbourhoods of Minimal $2$-Spheres}
This section is devoted to discussing stability of minimal surfaces and neighbourhoods of embedded minimal $2$-spheres with desirable geometric properties.\\
\begin{definition}[Stability Operator, Stability and Degeneracy]\label{stability}
Let $\Sigma$ be an orientable hypersurface embedded in a Riemannian manifold $(M,g)$ with $\nu$ a choice of unit normal. The \emph{stability operator} of $\Sigma$ is defined by the formula
\begin{align}
    L_{\Sigma} = -\Delta_{\Sigma} - |A_{\Sigma}|^2-Ric(\nu,\nu),
\end{align}
where $\Delta_{\Sigma}$, $A_{\Sigma}$ are the Laplace-Beltrami operator and the second fundamental form of $\Sigma$, respectively.\\
A minimal hypersurface $\Sigma$ is \emph{stable} if
\begin{align}
    \int_{\Sigma}\phi L_{\Sigma}\phi \geq 0, \text{ } \forall \phi \in C^{\infty}(\Sigma,\R).
\end{align}
$\Sigma$ is called strictly stable if the above quantity is strictly positive for all $\phi \neq 0$.\\
$\Sigma$ is said to be \emph{degenerate} stable if $ker(L_{\Sigma}) \neq \{0\}$.\\
We say that $\Sigma$ is unstable if it is not stable.\\
\end{definition}
We always assume that $\Sigma$ has empty boundary. For any smooth function $\phi: \Sigma \to \R$, we consider the family of surfaces given by 
\begin{align}
    \Sigma_t^{\phi} = \{\exp_x(t\phi(x)\nu(x)):x \in \Sigma\},
\end{align}
where $t$ lies in a sufficiently small time interval.\\ We now recall the first and second variation of area formulas (see, e.g. \cite{white})
\begin{proposition}[Variations of Area]\label{variation of area}
The first variation of area formula is
\begin{align}
    \frac{d}{dt}\biggr\rvert_{t=0}|\Sigma_t^{\phi}| = \int_{\Sigma}div_{\Sigma}(\phi \nu)=-\int_{\Sigma}H\phi,
\end{align}
where $H$ is the mean curvature of $\Sigma$.\\
If $\Sigma$ is minimal, then the second variation of area formula is
\begin{align}
    \frac{d^2}{dt^2}\biggr\rvert_{t=0}|\Sigma^{\phi}_t| = \int_{\Sigma}\phi L_{\Sigma}\phi.
\end{align}
\end{proposition}

\subsection{Contracting, Expanding and Mixed Neighbourhoods}

We now focus on our Riemannian $3$-sphere $(M,g)$.
\begin{definition}[c.f. Song \cite{song}]\label{geo nbd}
Let $\Sigma$ be an embedded minimal $2$-sphere in $M^3$. A neighbourhood $U$ of $\Sigma$ is called: \\
\begin{itemize}
    \item a \emph{contracting neighbourhood} if there exists $\delta > 0$ and a foliation $\{\Sigma_t\}_{-\delta<t<\delta}$ of $U$ by $2$-spheres with $\Sigma_0 = \Sigma$ such that the mean curvature vector of $\Sigma_t$ points towards $\Sigma$, for $0<|t|<\delta$.
    \item an \emph{expanding neighbourhood} if there exists $\delta > 0$ and a foliation $\{\Sigma_t\}_{-\delta<t<\delta}$ of $U$ by $2$-spheres with $\Sigma_0 = \Sigma$ such that the mean curvature vector of $\Sigma_t$ points away from $\Sigma$, for $0<|t|<\delta$.
    \item a \emph{mixed neighbourhood} if there exists $\delta > 0$ and a foliation $\{\Sigma_t\}_{-\delta<t<\delta}$ of $U$ by $2$-spheres with $\Sigma_0 = \Sigma$ such that the mean curvature vector of $\Sigma_t$ points towards (resp. away from) $\Sigma$ for $0<t<\delta$ and points away from (resp. towards) $\Sigma$ for $-\delta<t<0$.
\end{itemize}
\end{definition}
Before discussing which types of neighbourhoods arise, we recall the following standard lemma (see, e.g. \cite{padilla}):
\begin{lemma}\label{std lemma}
Let $L$ be an elliptic self-adjoint second-order differential operator on a compact manifold $N$ without boundary. Then, the minimal eigenvalue of $L$ is simple. Moreover, there exists a positive eigenfunction $w: N \to \R_{>0}$ such that the corresponding eigenspace is $\{\lambda \cdot w: \lambda \in \R\}$.
\end{lemma}
We first deal with the easy case where $\Sigma$ is non-degenerate:
\begin{proposition}\label{strictly stable}
Let $\Sigma$ be an embedded minimal $2$-sphere in $M^3$. If $\Sigma$ is strictly stable or unstable, then $\Sigma$ admits a contracting or expanding neighbourhood, respectively.
\end{proposition}
\begin{proof}[Proof of Proposition \ref{strictly stable}]
Let $\nu$ be a choice of unit normal to $\Sigma$ in $M^3$ and let $\phi$ be the lowest eigenfunction (which can be taken to be positive by Lemma \ref{std lemma}) of the stability operator $L_{\Sigma}$ of $\Sigma$ with eigenvalue $\lambda$. Assume, moreover, that $\phi$ is normalized, i.e. $||\phi||_{L^2} = 1$.\\
For $0 \leq |t| < \epsilon$ for $\epsilon > 0$ sufficiently small, we obtain a family of surfaces
\begin{align}
    \Sigma_t = \{\exp_x(t\phi(x)\nu(x)):x \in \Sigma\}
\end{align}
which foliate the tubular neighbourhood $\bigcup\limits_{t \in (-\epsilon,\epsilon)}\Sigma_t$ of $\Sigma$.\\
We can Taylor expand $H_{\Sigma_t}$ as
\begin{align}
    H_{\Sigma_t} &=-tL_{\Sigma}\phi + \mathcal{O}(t^2)\\
    &=-t\lambda \phi + \mathcal{O}(t^2).
\end{align}
Here, our convention is that the mean curvature vector is $\vv{H} = H\nu$, where $H$ is the mean curvature.\\
If $\Sigma$ is strictly stable, then $\lambda>0$. Thus, our expression for $H_{\Sigma_t}$ yields that $H_{\Sigma_t}<0$ for $t \in (0,\epsilon)$ and $H_{\Sigma_t} > 0$ for $t \in (-\epsilon,0)$. Hence, the neighbourhood
\begin{align}
    \bigcup_{t \in (-\epsilon,\epsilon)}\Sigma_t \supset \Sigma
\end{align}
is a contracting neighbourhood of $\Sigma$. 
Similarly, if $\Sigma$ were strictly unstable, then $\lambda < 0$: in this case, the same neighbourhood would be an expanding neighbourhood.\\
\end{proof}
Note that Proposition \ref{strictly stable} only tackles the case for a strictly stable (or unstable) minimal $2$-sphere. The following theorem (adapted from \cite{song}) addresses the case that $\Sigma$ is degenerate stable.

\begin{theorem}\label{song}
Let $\Sigma$ be a degenerate stable embedded minimal $2$-sphere in $M^3$. Then, at least one of the following holds:
\begin{enumerate}
    \item $M^3$ contains infinitely-many embedded minimal $2$-spheres.
    \item $\Sigma$ admits a contracting, expanding or mixed neighbourhood.
\end{enumerate}
\end{theorem}

\begin{proof}[Proof of Theorem \ref{song}]
The stability operator $L_{\Sigma} = -\Delta_{\Sigma}-|A_{\Sigma}|^2 - Ric(\nu,\nu)$ is an elliptic operator on $\Sigma$, a compact manifold without boundary. Notice also that $0$ is the minimal eigenvalue of $L_{\Sigma}$. Indeed, since $\Sigma$ is degenerate, there exists a non-trivial solution to $L_{\Sigma}\phi = 0$.\\
Thus, by Lemma \ref{std lemma}, there exists a smooth positive function $\phi_0:\Sigma \to \R_{>0}$ such that  $ker(L_{\Sigma}) = \{\lambda \cdot \phi_0: \lambda \in \R \}$. 
We will now seek to apply the implicit function theorem to prove our theorem. 
We consider the space
\begin{align}
    C^{2,\alpha}_0(\Sigma) = \biggr\{f \in C^{2,\alpha}(\Sigma) : \int_{\Sigma}f\phi_0 = 0\biggr\}
\end{align}
and the map
\begin{align}
    N:C^{2,\alpha}_0(\Sigma)\times \R \times \R \to C^{\alpha}(\Sigma),\text{ } (v,c,t) \mapsto H_{t\phi_0+v}-c,
\end{align}
defined near $(v,t) = (0,0)$. Here, $H_f$ denotes the mean curvature of the exponential graph of $f$ over $\Sigma$, i.e. if $f:\Sigma \to \R$ is a map with sufficiently small $C^{2,\alpha}$-norm, then
\begin{align}
    H_f = H\left( \{exp_x(f(x)\nu(x)):x \in \Sigma \} \right).
\end{align}
Since $H(\Sigma) = H_0 = 0$, we have
\begin{align}
    H_{\epsilon f} = -\epsilon L_{\Sigma}f + \mathcal{O}(\epsilon^2).
\end{align}
We now check that the map $N$ satisfies the assumptions of the implicit function theorem. Clearly, we have $N(0,0,0) = 0$. We also easily see that $N$ is $C^1$.\\
\newline
We consider the linearization $L:C^{2,\alpha}_0(\Sigma)\times \R \to C^{\alpha}(\Sigma)$ of $N$ given by:
\begin{align}
    L(v,c) &= \frac{d}{d \epsilon}\biggr\rvert_{\epsilon = 0}N(\epsilon v,\epsilon c, 0)\\
    &=\frac{d}{d \epsilon}\biggr\rvert_{\epsilon = 0}\left(H_{\epsilon v} -\epsilon c\right)\\
    &=-L_{\Sigma}v - c.
\end{align}
\begin{claim}\label{claim bij}
$L: C^{2,\alpha}_0(\Sigma)\times \R \to C^{\alpha}(\Sigma)$ is bijective.
\end{claim}
\begin{proof}[Proof of Claim \ref{claim bij}]
We have 
\begin{align}
    L(v,c) = 0 &\iff L_{\Sigma}v = -c\\
    &\implies \int_{\Sigma}\phi_0L_{\Sigma}v = -\int_{\Sigma}c\phi_0\\
    &\implies c = 0,
\end{align}
using the fact that $L_{\Sigma}$ is self-adjoint and $L_{\Sigma}\phi_0 = 0$ with $\phi_0 > 0$.\\ \newline
By Lemma \ref{std lemma}, the kernel of $L_{\Sigma}$ is one-dimensional. Thus, if $L(v,0)=0$, since $v \in C^{2,\alpha}_0(\Sigma)$, we must have that $v=0$. This shows that $L(v,c) = 0 \implies (v,c) = 0$, i.e. $L$ is an injective map.\\
\newline
We now claim that $L$ is surjective. Indeed, let $f \in C^{\alpha}(\Sigma)$. By the Fredholm alternative, since $L_{\Sigma}$ is an elliptic, self-adjoint operator, we can solve $-L_{\Sigma}v = f$ if and only if $f \perp  ker(L_{\Sigma})$, i.e. 
\begin{align}
    \int_{\Sigma}f\phi_0 = 0.
\end{align}
We now want to solve the equation $L(v,c) = f$ for $(v,c) \in C^{2,\alpha}_0(\Sigma)\times \R$. We choose $c$ as follows:
\begin{align}
    c:= -\frac{\int_{\Sigma}f\phi_0}{\int_{\Sigma}\phi_0}.
\end{align}
Note that this is well-defined since $\phi_0>0$. Let $\tilde{f}:= f+c$. By our choice of $c$, we have
\begin{align}
    \int_{\Sigma}\tilde{f}\phi_0=0.
\end{align}
Thus, we can find $v \in C^{2,\alpha}_0(\Sigma)$ such that $-L_{\Sigma}v = f+c$. Equivalently, we can solve $L(v,c) = f$ for our choice of $c$. This shows that $L$ is surjective and thus bijective.\\
\end{proof}
By the implicit function theorem, there exists $\delta > 0$ and a $C^1$ map $g:(-\delta,\delta) \to C^{2,\alpha}_0(\Sigma)\times \R$ with $g(0) = 0$ such that: 
\begin{align}
    N(g(t),t) = 0,
\end{align}
for each $t \in (-\delta,\delta)$. Write $g(t) = (v_t,c_t)$, where $v_t \in C^{2,\alpha}_0(\Sigma)$, $c_t \in \R$ with $v_0 = 0$ and $c_0 = 0$.\\
Thus, we have
\begin{align}
    H_{t\phi_0 + v_t} = c_t,
\end{align}
for each $t \in (-\delta,\delta)$. We define the function $\omega: \Sigma \times (-\delta,\delta) \to \R$, $\omega(x,t) = t\phi_0(x) + v_t(x)$. Hence, we have
\begin{align}
    H_{\omega(\cdot,t)} = c_t.
\end{align}
We also define the family of hypersurfaces
\begin{align}
    \Sigma_t = \{exp_x\left(\omega(x,t)\nu(x)\right):x \in \Sigma \}.
\end{align}
Note that $t=0$ is a zero of $c(t)$ since $c(0) = 0$ by the implicit function theorem construction.\\
There are two cases to distinguish here.\\
If $t=0$ is not an isolated zero of $c(t)$, then there exist infinitely-many times $t \in (-\delta,\delta)$ such that $\Sigma_t$ are embedded minimal $2$-spheres which proves the theorem.\\
On the other hand, if $t=0$ is an isolated zero of $c(t)$, then there exists a $0<\delta_1\leq \delta$ such that $c(t) \neq 0$ for $t \in (-\delta_1,\delta_1)$ with $t \neq 0$.\\
Thus:
\begin{align}
    U:= \bigcup_{-\delta_1<t<\delta_1}\Sigma_t
\end{align}
is a neighbourhood of $\Sigma$ of one of the three types (i.e. contracting, expanding or mixed).\\
\end{proof}
\subsection{Squeezing Maps}
In this subsection, we construct a family of 'squeezing maps'. \\ \newline
Let $(M,g)$ be a Riemannian $3$-sphere and suppose that $\Sigma$ is an embedded minimal $2$-sphere admitting a contracting or mixed neighbourhood foliated by $2$-spheres $\{\Sigma_s\}_{-1<s<1}$, where the surfaces indexed by $s \in [0,1)$ are the contracting half and 
\begin{align}
    \Sigma_s = \{exp_x(\omega(x,s)\nu(x)):x \in \Sigma \}.
\end{align}
We define a smooth map $f$ on the contracting half by $exp_x(\omega(x,s)\nu(x))\mapsto s$. By construction, $\nabla f \neq 0$ and $\Sigma_s = f^{-1}(s)$ for $s \in [0,1)$. Moreover, $\Sigma_s = f^{-1}(s)$ is strictly mean-concave for $s> 0$, i.e.
\begin{align}
    \langle \nabla f, \vv{H}(\Sigma_s)\rangle < 0.
\end{align}
From $f$, we obtain the vector field $X = \frac{\nabla f}{|\nabla f|^2}$. Let $\phi:\Sigma \times [0,1] \to M$ be the smooth embedding defined by:
\begin{align}
    \begin{cases} \frac{\partial \phi}{\partial s}(x,s) =X(\phi(x,s))\\ \phi(x,0)=x. \end{cases}
\end{align}
By a standard computation, we obtain that $\frac{d}{ds}f(\phi(x,s)) = 1$ for $s \in [0,1)$: in particular, this yields that $\phi(\Sigma,s) = \Sigma_s$ by recalling that $\Sigma_s = f^{-1}(s)$. Now, let $\Omega_r = \phi(\Sigma \times [0,r))$ and consider the map $P: \Omega_1\times [0,1] \to \Omega_1$ defined by 
\begin{align}
    P(\phi(x,s),t) = \phi(x,(1-t)s).
\end{align}
From this map, we obtain a one-parameter family of maps $P_t:\Omega_1 \to \Omega_1$ defined by $P_t(\phi(x,s)):=P(\phi(x,s),t)$.
\begin{theorem}[cf. Marques-Neves {{\cite[Proposition 5.7]{neves}}}]\label{squeezing}
There exists $r_0>0$ such that $P_t:\Omega_{r_0} \to \Omega_{r_0}$ satisfies:
\begin{enumerate}[(i)]
    \item $P_0(x) = x$, for all $x \in \Omega_{r_0}$ and $P_t(x) = x$, for all $x \in \Sigma$ and $0 \leq t \leq 1$,
    \item $P_t(\Omega_r)\subset \Omega_r$ for $0 \leq t \leq 1$, $r \leq r_0$ and $P_1(\Omega_{r_0}) = \Sigma$,
    \item $P_t:\Omega_{r_0} \to \Omega_{r_0}$ is an embedding, for $0 \leq t < 1$,
    \item For all surfaces $V \subset \Omega_{r_0}$, we have
    \begin{align}
        \frac{d}{dt}|P_t(V)| \leq 0,
    \end{align}
    with equality if and only if $V \subset \Sigma$.
\end{enumerate}
\end{theorem}
\begin{proof}[Proof of Theorem \ref{squeezing}]
By definition, properties ($i$), ($ii$) and ($iii$) hold for any $r_0 > 0$.\\
We now prove property ($iv$). If $y=\phi(x,s) \in \Omega_1$, then a standard computation yields that $\frac{d}{dt}P_t(y) = Z_t(P_t(y))$, where $Z_t = -\frac{f}{1-t}X$. Given $x \in \Omega_1$, we consider a $2$-dimensional subspace $\sigma \subset T_x \Omega_1$ and write:
\begin{align}
    div_{\sigma}Z_t(x) &=\langle \nabla_{v_1}Z_t,v_1\rangle + \langle \nabla_{v_2}Z_t,v_2\rangle,
\end{align}
where $\{v_1,v_2\}$ is an orthonormal basis of $\sigma$. By the first variation of area formula, we have
\begin{align}
    \frac{d}{dt}|P_t(V)| = \int_{P_t(\Sigma)} div_{\sigma}Z_t.
\end{align}

We can choose coordinates $(x_1,x_2)$ near $x \in \Sigma$. Then, $\biggr\{e_i  = \frac{\partial \phi}{\partial x_i}\biggr\}_{i=1}^2$ is a frame of $T\Sigma_s = Tf^{-1}(s)$ where $\Sigma_s$ are the $2$-spheres which foliate the contracting half of the neighbourhood of $\Sigma$. Let $N = \frac{\nabla f}{|\nabla f|}$ and let $A = A_{\Sigma_s}$ be the second fundamental form of $\Sigma_s$ (with respect to the unit normal $N$).\\
\begin{claim}
The following hold:
\begin{enumerate}[(i)]
    \item $\langle \nabla_{e_i}Z_t,e_j\rangle = -\langle Z_t,A(e_i,e_j)\rangle$, \\
    \item $\langle \nabla_{e_i}Z_t,N\rangle = -\langle \nabla_NZ_t,e_i\rangle$,\\
    \item $\langle \nabla_NZ_t,N\rangle = -\frac{1}{1-t}(1+f\langle \nabla_NX,N\rangle)$.
\end{enumerate}
\end{claim}
\begin{proof}[Proof of Claim 2]
Since $Z_t$ is a multiple of the unit normal vector $N$, we have:
\begin{align}
    \langle Z_t,A(e_i,e_j)\rangle &=\langle Z_t,\nabla_{e_i}e_j\rangle \\
    &=-\langle \nabla_{e_i}Z_t,e_j\rangle,
\end{align}
where we used that $Z_t$ is orthogonal to $e_j$. This proves ($i$).\\
Next, we compute:
\begin{align}
    \langle \nabla_{e_i}Z_t,N\rangle &= \left\langle \nabla_{e_i}\left(-\frac{f}{1-t}X\right),N\right\rangle\\
    &=-\frac{f}{1-t}\langle \nabla_{e_i}X,N\rangle\\
    &=\frac{f}{1-t}\left\langle \frac{\partial \phi}{\partial x_i},\nabla_{\frac{\partial \phi}{\partial s}}N\right\rangle,
\end{align}
where we used the fact that $\langle e_i,N \rangle = 0$. Similarly, we have:
\begin{align}
    \langle \nabla_NZ_t,e_i \rangle &=-\frac{f}{1-t}\left\langle \nabla_{\frac{\partial \phi}{\partial s}}N,\frac{\partial \phi}{\partial x_i}\right\rangle,
\end{align}
which, together with the above, proves ($ii$).\\
Finally, a similar computation yields ($iii$) by noting that $N(f) = |\nabla f|$.
\end{proof}
Now, continuing the proof of the theorem, we choose an orthonormal basis $\{v_1,v_2\}$ for $\sigma$ so that $v_1$ is orthogonal to $N$. Then, we have $v_2 = (\cos \theta)u + (\sin \theta)N$, where $u \in T_y\Sigma_s$, $\theta$ is some polar angle and $|u| = 1$.\\
Using Claim 2, we now compute:
\begin{align}
    div_{\sigma}Z_t &=\langle \nabla_{v_1}Z_t,v_1\rangle + \langle \nabla_{v_2}Z_t,v_2\rangle \\
    &=-\langle Z_t,\vv{H}(\Sigma_s)\rangle + \sin^2\theta(\langle \nabla_NZ_t,N\rangle - \langle \nabla_uZ_t,u\rangle)\\
    &=\frac{1}{1-t}\left(|\nabla f|^{-1}s \langle N,\vv{H}(\Sigma_s)\rangle-\sin^2\theta (1+s\langle \nabla_NX,N\rangle + s\langle X,A(u,u)\rangle\right).
\end{align}
Since $\Sigma_s$ are strictly mean-concave for $s>0$, we have $ \langle N,\vv{H}(\Sigma_s)\rangle < 0 $. Thus, for $s$ sufficiently small, we have that $div_{\sigma}Z_t$ is non-positive with equality if and only if $s=0$.\\
\end{proof}
As a corollary of the proof of Theorem \ref{squeezing}, we obtain
\begin{corollary}\label{expanding}
If $\Sigma$ admits an expanding neighbourhood, then we get the existence of a collection of 'expanding maps', i.e.
\begin{align}
    \frac{d}{dt}|P_t(V)| > 0,
\end{align}
with equality if and only if $V \subset \Sigma$.
\end{corollary}
\bigskip
\section{Proof of the Main Theorem}
The goal of this section is to prove our main theorem, which we restate here for convenience of the reader.
\begin{theorem}[Existence Theorem]\label{main}
Let $(M^3,g)$ be a Riemannian manifold diffeomorphic to $\mathbb{S}^3$. Then, $M$ contains at least two embedded minimal $2$-spheres or admits an optimal foliation by $2$-spheres. More precisely, either $(M^3,g)$
\begin{enumerate}[(1)]
    \item contains a strictly stable embedded minimal $2$-sphere, in which case $M$ contains at least three embedded minimal $2$-spheres,
    \item or contains a degenerate embedded minimal $2$-sphere $\Sigma$. In this case, either
    \begin{enumerate}[(a)]
        \item $M$ contains infinitely-many embedded minimal $2$-spheres.
        \item or $\Sigma$ admits either
        \begin{enumerate}[(i)]
            \item a contracting neighbourhood, in which case $M$ contains at least three embedded minimal $2$-spheres,
            \item or a mixed neighbourhood, in which case $M$ contains at least two embedded minimal $2$-spheres,
            \item or an expanding neighbourhood, in which case $M$ admits an optimal foliation by $2$-spheres.
        \end{enumerate}
    \end{enumerate}
    \item or contains an unstable embedded minimal $2$-sphere, in which case $M$ admits an optimal foliation by $2$-spheres and contains a second embedded minimal $2$-sphere.
\end{enumerate}
\end{theorem}
As mentioned in the introduction, there are three scenarios to distinguish: $M^3$ admits a degenerate stable embedded minimal $2$-sphere but no strictly stable minimal $2$-spheres; $M^3$ admits a strictly stable embedded minimal $2$-sphere; $M^3$ does not admit a stable embedded minimal $2$-sphere.\\

\begin{proof}[Proof of Theorem \ref{main}]
The bulk of the proof will consist of dealing with the scenario where $M^3$ admits a degenerate stable embedded minimal $2$-sphere $\Sigma$ but no strictly stable minimal $2$-spheres. We will tackle the other scenarios at the end.\\ Suppose now that $M^3$ contains an embedded degenerate stable minimal $2$-sphere $\Sigma$ but contains no strictly stable embedded minimal $2$-spheres. By Theorem \ref{song}, either there exist infinitely-many embedded minimal $2$-spheres (in which the case the theorem is proved) or $\Sigma$ admits a contracting, mixed or expanding neighbourhood.\\
In order to fix notation, we denote by $\{\Sigma_t\}_{t \in (-\delta,\delta)}$ the foliation of a tubular neighbourhood of $\Sigma$ by $2$-spheres (such a foliation is constructed in Theorem \ref{song}).\\
\begin{claim}\label{contracting claim}
If $\Sigma$ admits a contracting or mixed neighbourhood, then we can find another embedded minimal $2$-sphere in $M$.\\
\end{claim}
\begin{proof}[Proof of Claim \ref{contracting claim}]
We can decompose $M$ according to
\begin{align}
    M = S^- \cup \Sigma \cup S^+,
\end{align}
where $S^{\pm}$ are the connected components of $M \setminus \Sigma$ chosen so that $S^+$ contains the contracting half of the neighbourhood (without loss of generality, we assume that this corresponds to $\{\Sigma_t\}_{t \in (0,\delta)}$) and $S^-$ contains the other half (which is contracting in the case of a contracting neighbourhood and expanding in the case of a mixed neighbourhood). Then, $\overline{S^+}$ is a smooth $3$-disk with minimal boundary by construction. In particular, we have $\partial \overline{S^+} = \Sigma$.\\
In order to prove Claim \ref{contracting claim}, we will seek to apply the min-max theorem \cite[Theorem 10]{lio} from Ketover-Liokumovich-Song. For this, we recall the necessary definitions from \cite{lio}.\\

\begin{definition}[One-Parameter Sweepouts]\label{sweepout}
Let $(N,g)$ be a compact Riemannian $3$-manifold with connected boundary $\partial N$. Let $\{\tilde{\Sigma}_t\}_{t \in I}$ be a family of closed oriented surfaces in $N$ where $I=[a,b]$. \\ $\{\tilde{\Sigma}_t\}_{t \in I}$ is said to be a smooth one-parameter sweepout of $N$ if 
\begin{enumerate}[(i)]
    \item for each $t \in (a,b)$, $\tilde{\Sigma}_t$ is a smooth surface in $Int(N)$, 
    \item $\tilde{\Sigma}_t$ varies smoothly in $t \in [a,b)$,
    \item $\tilde{\Sigma}_a = \partial N$, $\tilde{\Sigma}_b$ is a one-dimensional graph and $\tilde{\Sigma}_t$ converges to $\tilde{\Sigma}_b$ in the Hausdorff topology, as $t \to b$.
\end{enumerate}
Assuming $\partial N$ is a $2$-sphere, one says that a smooth one-parameter sweepout of $N$ is a one-parameter sweepout of $N$ by $2$-spheres if $\tilde{\Sigma}_t$ is a smooth $2$-sphere for each $t \in [a,b)$.
\end{definition}
\begin{definition}[One-Width]\label{min-max}
Let $(N,g)$ be a compact $3$-manifold with connected boundary $\partial N$ and let $\Lambda$ be the set of all smooth $1$-parameter sweepouts of $N$. The width of $N$ is defined by the formula
\begin{align}
    \omega(N,\Lambda) := \inf_{\{\tilde{\Sigma}_t\} \in \Lambda} \sup_{t \in I}|\tilde{\Sigma}_t|.
\end{align}
A sequence $\{\tilde{\Sigma}^n_t\}$ of smooth sweepouts is called a \emph{minimizing sequence} if
\begin{align}
    \lim_{n \to \infty}\sup_{t \in [a,b]}|\tilde{\Sigma}^n_{t}| = \omega(N,\Lambda).
\end{align}
\end{definition}
Back to the problem at hand, let $\Lambda$ be the set of all smooth $1$-parameter sweepouts of $\overline{S^+}$ by $2$-spheres parametrized by the interval $I = [0,1]$ (note that $\Lambda$ is non-empty as we can construct such a sweepout on the Euclidean $3$-disk and pull it back by a diffeomorphism to this one). This yields the corresponding min-max width
\begin{align}
    \omega(\overline{S^+},\Lambda) = \inf_{\{\tilde{\Sigma}_t\} \in \Lambda}\sup_{t \in [0,1]}|\tilde{\Sigma}_t|.
\end{align}
By Theorem \ref{squeezing}, $\Sigma$ admits a squeezing map (i.e. a collection of maps $P_t$). By property ($iv$) of the squeezing maps and \cite[Lemma 8.1]{lio}, it follows that
\begin{align}
    \omega(\overline{S^+},\Lambda) &=\inf_{\{\tilde{\Sigma}_t\} \in \Lambda} \sup_{t \in [0,1]}|\tilde{\Sigma}_t| \\
    &>|\Sigma|.
\end{align}
Note that the existence of the squeezing maps from Theorem \ref{squeezing} allows us to apply the min-max theorem of Ketover-Liokumovich-Song. More precisely, by \cite[Theorem 10]{lio} (in particular, Remark 11 following Theorem 10), there exists a minimizing sequence $\Sigma^n_{t}$ in $\overline{S^+}$ which converges as varifolds to $\Sigma^{\infty}$, an embedded minimal $2$-sphere lying in $Int(\overline{S^+})$. This proves the claim.\\
\end{proof}
As a reminder, Claim \ref{contracting claim} proves the existence of another embedded minimal $2$-sphere in $M$ provided it admits a contracting or mixed neighbourhood.\\

Before discussing the case where it admits an expanding neighbourhood, we need to recall a few facts and definitions from \cite{haslhofer}.\\

As in the introduction, we would like to consider a space of embeddings of $\mathbb{S}^2$ into $\mathbb{S}^3$ together with certain permissible degenerations. More precisely, consider the spaces:
\begin{align}
    \mathcal{X} = \{\phi(\mathbb{S}^2) \text{ } |\text{ } \phi:\mathbb{S}^2 \to \mathbb{S}^3 \text{ is a smooth embedding} \}
\end{align}
and
\begin{align}
    \mathcal{Y} = \{\phi(\mathbb{S}^2) \text{ }| \text{ } \phi:\mathbb{S}^2 \to \mathbb{S}^3 \text{ is a smooth map whose image is a one-dimensional graph}\}.
\end{align}
As in \cite[Section 2]{haslhofer}, we equip $\mathcal{S}:=\mathcal{X} \cup \mathcal{Y}$ with the unparametrized smooth topology. By Hatcher's theorem (c.f. \cite{hatcher,bamler}), the space $\mathcal{S}$ is homotopy equivalent to $\R \mathbb{P}^4 \backslash B$, where $B$ is an open ball. Thus, the relative cohomology ring of $\mathcal{S}$ is given by 
\begin{align}
    H^*(\mathcal{S},\partial \mathcal{S},\Z_2) = \Z_2[\alpha]/(\alpha^5),
\end{align}
where $\alpha$ is a generator of $H^1(\mathcal{S},\partial \mathcal{S},\Z_2)$.\\
\begin{definition}[k-Width]\label{width}
If $\alpha$ is a generator of $H^1(\mathcal{S},\partial \mathcal{S},\Z_2)$, then the $k$-width of $M$ (for $k=1,2,3,4$) is defined by 
\begin{align}
    \omega_k(M) :=\inf_{\Phi^*\alpha^k \neq 0 }\sup_{x \in Dom(\Phi)}|\Phi(x)|.
\end{align}
Here, the infimum is taken over all continuous maps $\Phi:X \to \mathcal{S}$, where $X$ is some simplicial complex such that $\Phi^*\alpha^k \neq 0$.
\end{definition}
Finally, we recall the min-max theorem (see, e.g., \cite{haslhofer}):
\begin{theorem}[Min-Max Theorem]\label{min-max theorem}
 Let $M^3$ be a Riemannian manifold diffeomorphic to $\mathbb{S}^3$. Then, for each $i \in \{1,2,3,4\}$, there exists a minimizing sequence which converges to a stationary integral varifold
\begin{align}
    V_i = \sum_{j=1}^{k_i}m^j_i\Sigma^j_i,
\end{align}
where $\{\Sigma^j_i\}_{j=1}^{k_i}$ is a collection of pairwise-disjoint embedded minimal $2$-spheres and $m^j_i>0$ are integer multiplicites. Moreover:
\begin{align}
    |V_i| = \omega_i(M) = \sum_{j=1}^{k_i}m^j_i|\Sigma^j_i|
\end{align}
and
\begin{align}
    0<\omega_1(M) \leq \omega_2(M) \leq \omega_3(M) \leq \omega_4(M).
\end{align}
\end{theorem}
Finally, before stating our claim for expanding neighbourhoods, we have a final definition to recall:
\begin{definition}[Optimal Foliations, c.f. {{\cite[Definition 3]{haslhofer}}}]\label{optimal foliation}
A one-paramter family of sets $\{\hat{\Sigma}_t\}_{t \in [-1,1]}$ in $M$ is called an optimal foliation of $M$ by $2$-spheres if
\begin{enumerate}[(i)]
    \item $\hat{\Sigma}_t$ is a smooth embedded $2$-sphere for each $t \in (-1,1)$,
    \item $\hat{\Sigma}_{-1}$ and $\hat{\Sigma}_1$ are one-dimensional graphs,
    \item $\hat{\Sigma}_0$ is a minimal $2$-sphere realizing the $1$-width $\omega_1(M)$,
    \item $|\hat{\Sigma}_t|<|\hat{\Sigma}_0|$ for $t \neq 0$,
    \item $\hat{\Sigma}_t$ depends smoothly on $t \in (-1,1)$,
    \item $\hat{\Sigma}_t \to \hat{\Sigma}_{\pm 1}$ as $t \to \pm 1$ in the Hausdorff topology,
    \item $\hat{\Sigma}_s \cap \hat{\Sigma}_t = \emptyset$ whenever $s \neq t$.
\end{enumerate}
\end{definition}

\begin{remark}
In contrast to \cite[Definition 3]{haslhofer}, we do not actually obtain the inequality (derived via the second variation of area formula):
\begin{align}
    |\hat{\Sigma}_t| \leq |\hat{\Sigma}_0| - ct^2
\end{align}
in the degenerate case (the $t^2$-term is replaced by higher order terms in $t$), as discussed in the introduction (see Remark \ref{intro}).
\end{remark}

\begin{claim}\label{expanding claim}
If $\Sigma$ admits an expanding neighbourhood, then $M$ admits an optimal foliation by $2$-spheres.
\end{claim}

\begin{proof}[Proof of Claim \ref{expanding claim}]
If $\Sigma$ admits an expanding neighbourhood, then we can decompose $M$ as follows
\begin{align}
    M = D^- \cup N_{\delta}(\Sigma)\cup D^+,
\end{align}
where $D^{\pm}$ are the connected components of $M \setminus N_{\delta}(\Sigma)$ and are smooth $3$-disks with mean-convex boundary, and where $N_{\delta}(\Sigma)$ is a tubular neighbourhood of $\Sigma$ in $M$ given by 
\begin{align}
    N_{\delta}(\Sigma) = \{\exp_x(\omega(x,t)\nu(x)):x \in \Sigma,-\delta<t<\delta\},
\end{align}
whose form is guaranteed by the proof of Theorem 6.\\
By \cite[Theorem 1.8]{haslhofer} (which was proven using mean curvature flow with surgery), exactly one of the following holds:
\begin{enumerate}[(a)]
    \item there exists a stable embedded minimal $2$-sphere in $Int(D^+)$ or in $Int(D^-)$,
    \item there exist smooth foliations $\{ \Sigma^{\pm}_t \}_{t \in [-1,1]}$ of $D^{\pm}$ by mean-convex $2$-spheres.
\end{enumerate}
If we are in case (a), we obtain at least $2$ embedded minimal $2$-spheres and we are done.\\
Suppose now that we are in case (b). Note that we may assume from now on that $M^3$ contains no stable embedded minimal $2$-spheres with contracting or mixed neighbourhoods and that $M$ contains only finitely-many embedded minimal $2$-spheres (as the theorem would be proven otherwise). \\

Recall that we have a smooth foliation $\{\Sigma^{\pm}_t\}_{t \in [-1,1]}$ of $D^{\pm}$ by mean-convex $2$-spheres and an expanding tubular neighbourhood $N_{\delta}(\Sigma)$ of $\Sigma$ which is foliated by $\{\Sigma_t\}_{t \in (-\delta,\delta)}$.\\
Thus, concatenating these foliations (up to relabelling time) yields an optimal foliation $\{\hat{\Sigma}_t\}_{t \in [-1,1]}$ of $M$ by $2$-spheres. All properties are clear, exept for ($iii$) which we will show in Lemma \ref{infimum}. We now adapt several lemmas from \cite{haslhofer}, carefully excluding certain cases which were addressed previously:
\begin{lemma}[Lemma 3.4 of \cite{haslhofer}, modified]\label{intersect}
Let $M^3$ be a $3$-sphere containing only finitely-many embedded minimal $2$-spheres and no stable embedded minimal $2$-spheres with a contracting or mixed neighbourhood. Then, any two embedded minimal $2$-spheres intersect.
\end{lemma}
\begin{proof}[Proof of Lemma \ref{intersect}]
We proceed by contradiction, i.e. suppose that $\Sigma_1$ and $\Sigma_2$ are two embedded minimal $2$-spheres that do not intersect.\\
By assumption and by Proposition \ref{strictly stable}, $\Sigma_1$ must have an expanding neighbourhood. Then, there exists a tubular neighbourhood of $\Sigma_1$ which is foliated by mean-convex $2$-spheres (asides from $\Sigma_1$). We use this foliation to move $\Sigma_1$ towards to $\Sigma_2$ to produce a mean-convex $\tilde{\Sigma}$.\\
Note that we can write
\begin{align}
    M = S^- \cup N_{\epsilon}(\tilde{\Sigma}) \cup S^+,
\end{align}
where $\epsilon > 0$ is sufficiently small and $N_{\epsilon}(\tilde{\Sigma})$ is some $\epsilon$-tubular neighbourhood, and $S^{\pm}$ are the connected components of $M \setminus N_{\epsilon}(\tilde{\Sigma})$. Choose $S^+$ so that $S^+$ contains $\Sigma_2$.\\ 
We now flow $\tilde{\Sigma}$ by level set flow to get $\{\tilde{\Sigma}_t\}_{t \in [0,T)}$.\\
By the work of White \cite{white1,white2,white5}, either (a) the flow converges to finitely many stable embedded minimal $2$-spheres and produces a mixed or contracting neighbourhood of them, or (b) the flow becomes extinct in finite time and produces a possibly singular foliation of $\overline{S^+}$. Now, (a) is exluded by the assumption of our lemma. Hence, we are in case (b) and there exists a time $t_1$ such that $\tilde{\Sigma}_{t_1} \cap \Sigma_2 \neq \emptyset$. This contradicts the avoidance principle.
\end{proof}
Consider the quantity
\begin{align}
    \gamma(M) = \inf_{\Sigma' \in \mathcal{S}_{\text{min}}}|\Sigma'|,
\end{align}
where $\mathcal{S}_{\text{min}}$ is the space of embedded minimal $2$-spheres in $M$.\\
\begin{lemma}[Lemma 3.2 of \cite{haslhofer}, modified]\label{infimum}
Under the assumptions of Lemma \ref{intersect}, there exists an embedded minimal $2$-sphere $\Sigma'$ in $M$ with multiplicity $1$ which realizes the infimum $\gamma(M)$. Moreover, we have:
\begin{align}
    \gamma(M) = \omega_1(M).
\end{align}
\end{lemma}
\begin{proof}[Proof of Lemma \ref{infimum}]
The set $\mathcal{S}_{\text{min}}$ is always non-empty by the Simon-Smith existence theorem. \\
By the monotonicity formula, we obtain a lower bound on the area of minimal surfaces in $M$: this implies that $\gamma(M) > 0$. Since $\mathcal{S}_{\text{min}}$ is a finite set, we can choose some $\Sigma'$ which realizes the infimum $\gamma(M)$. \\
We now show that $\gamma(M) = \omega_1(M)$. We first have that $\gamma(M) \leq \omega_1(M)$. Suppose otherwise, i.e. suppose that $\gamma(M) > \omega_1(M)$. By the min-max theorem (Theorem 9), we can produce an embedded minimal $2$-sphere $\tilde{\Sigma}$ in $M$ with $\tilde{\Sigma} = \omega_1(M)$ - however, this would contradict the definition of $\gamma(M)$. We also have that $\omega_1(M) \leq \gamma(M)$. Indeed, repeating the argument from pages 15-16 applied to $\Sigma'$, we can produce some foliation $\{\Sigma_t'\}_{t \in [-1,1]}$ of $M$ with $\Sigma_0' = \Sigma'$ satisfying all the properties from Definition \ref{optimal foliation} except for ($iii$). In particular, $|\Sigma_t'|<|\Sigma_0'|$ for $t \neq 0$. By the definition of the $1$-width, this yields that $\omega_1(M) \leq |\Sigma_0'| = \gamma(M)$ as required.
\end{proof}

We have thus shown that $M$ admits an optimal foliation by $2$-spheres.
\end{proof}
We have now proven the main theorem when $M$ admits a degenerate stable embedded minimal $2$-sphere and admits no strictly stable embedded minimal $2$-spheres. \\

Finally, let us deal with the remaining scenarios:\\
\newline
First, we suppose that $M$ contains a strictly stable embedded minimal $2$-sphere $\Sigma$. By Proposition \ref{strictly stable}, $\Sigma$ has a contracting neighbourhood. Then, by applying the Claim \ref{contracting claim} argument (i.e. applying \cite[Theorem 10]{lio}) to both contracting halves, we can find two further embedded minimal $2$-spheres in $M$: this produces a total of $3$ embedded minimal $2$-spheres in $M$ which proves the theorem.\\ \newline

Now, suppose that $M$ contains an unstable embedded minimal $2$-sphere $\Sigma$. Then, by Proposition \ref{strictly stable}, $\Sigma$ has an expanding neighbourhood. By Claim \ref{expanding claim}, we can produce an optimal foliation of the manifold $M$ by $2$-spheres.\\
Using this foliation, by \cite[Theorem 4.1]{haslhofer} (which is a consequence of the catenoid estimate from \cite{ketover}), we can construct a $2$-parameter sweepout $\{\Gamma_t\}_{t \in \R \mathbb{P}^2}$ detecting $\alpha^2$ such that $\sup_{t \in \Gamma_t}|\Gamma_t| < 2|\Sigma_0|$. In particular, we obtain that $\omega_2(M) < 2\omega_1(M)$. By the min-max theorem (Theorem 9), we obtain stationary integral varifolds $V_1$, $V_2$ (associated to the family detecting $\alpha$, $\alpha^2$, respectively) given by
\begin{align}
    V_i = \sum_{j=1}^{k_i}m_i^j\Sigma^j_i,
\end{align}
where $\Sigma^j_i$ are embedded minimal $2$-spheres which are pairwise disjoint, i.e. $\Sigma^j_i \cap \Sigma^{j'}_i = \emptyset$ whenever $j \neq j'$. Moreover, we have $|V_i| = \omega_i(M)$.\\
By Lemma \ref{intersect}, any two embedded minimal $2$-spheres intersect (under the appropriate assumptions). Thus, we must have that $V_1 = m_1 \Sigma_1$ and $V_2 = m_2 \Sigma_2$. By Lemma \ref{infimum}, we must have that $m_1 = 1$, which yields that $|\Sigma_1| = \omega_1(M) = \gamma(M)$. We have that
\begin{align}
    m_2|\Sigma_2| = \omega_2(M) < 2\omega_1(M) = 2|\Sigma_1|.
\end{align}
Since $|\Sigma_1| = \gamma(M)$, we obtain that $m_2 = 1$: this shows that $V_1 = \Sigma_1$, $V_2 = \Sigma_2$, i.e. both varifolds are embedded minimal $2$-spheres with multiplicity one.\\
If $\omega_1(M) \neq \omega_2(M)$, then $\Sigma_1 \neq \Sigma_2$ and we are done.\\ \newline
Suppose, now, that $\omega_1(M) = \omega_2(M)$. Then, by a Lusternik-Schnirelmann argument (\cite[Theorem 5.2]{haslhofer}), there exist infinitely-many embedded minimal $2$-spheres of area $\omega_1(M)$ and we are again done.\\
This proves the theorem in the case of the manifold containing an unstable embedded minimal $2$-sphere.

\end{proof}

\bigskip
\textit{Department Of Mathematics, University of Toronto, Toronto, ON, M5S2E4, Canada}\\
\emph{E-mail Address: salim.deaibes@mail.utoronto.ca}

\end{document}